\documentclass[letterpaper,11pt]{amsart}
\usepackage{amsfonts,amsmath,amsthm,graphicx,amssymb,amscd}
\usepackage[small,nohug,heads=littlevee]{diagrams} 
\diagramstyle[labelstyle=\scriptstyle]
\theoremstyle{definition}
\newtheorem{thm}{Theorem}[section]
\newtheorem{cor}[thm]{Corollary}
\newtheorem{lemma}[thm]{Lemma}
\newtheorem{prop}[thm]{Proposition}
\newtheorem{defn}[thm]{Definition}
\newtheorem{example}[thm]{Example}
\newtheorem{rem}[thm]{Remark}
\newtheorem{question}[thm]{Question}
\newcommand{\abs}[1]{\left\vert#1\right\vert}
\newcommand{\ideal}[1]{\left\langle#1\right\rangle}
\newcommand{\set}[1]{\left\{#1\right\}}
\newcommand{\R}{\mathbb R}
\newcommand{\Z}{\mathbb Z}

\newcommand{\Q}{\mathbb Q}
\newcommand{\C}{\mathbb C}

\newcommand{\A}{\mathcal A}
\newcommand{\B}{\mathcal B}
\newcommand{\Hh}{\mathcal H}

\DeclareMathOperator{\gr}{gr}
\begin{document}

\title[Eq. Cohomology \& V.G. Filtration]{Equivariant Cohomology and the Varchenko-Gelfand Filtration}
\author[D. Moseley]{Daniel Moseley}
\begin{abstract}
The cohomology of the configuration space of $n$ points in $\R ^3$ is isomorphic to the regular representation of the symmetric group, which acts by permuting the points.  We give a new proof of this fact by showing that the cohomology ring is canonically isomorphic to the associated graded of the Varchenko-Gelfand filtration on the cohomology of the configuration space of $n$ points in $\R ^1$.  Along the way, we give a presentation of the equivariant cohomology ring of this space with respect to a circle acting on $R^3$ via rotation around a fixed line.  We extend our results to the settings of arbitrary real hyperplane arrangements (the aforementioned theorems correspond to the braid arrangement) as well as oriented matroids.
\end{abstract}
\maketitle

\section{Introduction}

\begin{defn}
  Let $V$ be a finite dimensional real vector space, and let $\mathcal{A} = \set{H_1, \ldots, H_n}$ be a hyperplane arrangement in $V$ given by $H_i = \omega_i^{-1}(0)$ for some non-constant affine linear form $\omega_i :V \to \R$. Let $\overline{\omega_i}$ be the associated linear map. Define affine linear maps $\omega_{i,k}: V^k \to \R^k$ by 
  \[ \omega_{i,k}(v_1, \ldots, v_k) =(\omega_i(v_1),\overline{\omega_i} (v_2), \ldots, \overline{\omega_i}(v_k)). \]
  The space $M_k(\mathcal{A})$ is defined to be the complement of the union of the affine subspaces 
  \[ \omega_{i,k}^{-1}(0, 0,\ldots, 0). \]
\end{defn}

When $k=1$, this is just the complement of the arrangement. When $k=2$, this is isomorphic to the complement of the complexified arrangement. When $\mathcal{A}$ is the braid arrangement $\mathcal{B}_n$, $M_k(\mathcal{A})$ is the configuration space of $n$ ordered points in $\R^k$.

Consider the ring\footnote{All cohomology rings in this paper will be taken with coefficients in $\Q$.} $H^0(M_1(\mathcal{A}))$ of locally constant functions on $M_1(\mathcal{A})$. Varchenko and Gelfand defined a filtration of this ring via Heaviside functions. Let 
\[ H_i ^+ = \set{v\in V \mid \omega_i (v) >0},\]
and let 
\[ H_i ^- = \set{v\in V \mid \omega_i (v) <0}. \] 
Define the Heaviside function $x_i \in H^0(M_1(\mathcal{A}))$ by putting 
\[ x_i(v) = \left\{ \begin{matrix} 1 & v\in H_i ^+ \\ 0 & v\in H_i ^- . \end{matrix} \right. \]

\begin{prop}\label{VGpres}\cite[Thm 4.5]{VG}
  Consider the map $\psi: \Q[e_1, \ldots, e_n] \to H^0(M_1(\mathcal{A}))$ taking $e_i$ to $x_i$, and let $\mathcal{I}_1$ be the kernel. This map is surjective, and  $\mathcal{I}_1$ is generated by the following relations:
  \[\begin{array}{l l}
    (1) & e_i ^2 - e_i \\
    \\
    (2) & \displaystyle\prod_{i\in S^+} e_i \prod_{j\in S^-} (e_j-1)  \text{ if} \displaystyle \bigcap_{i\in S^+} H_i ^+ \cap \bigcap_{j\in S^-} H_j ^- = \varnothing. \hspace{.8in}\\
    \\
    (3) & \displaystyle\prod_{i\in S^+} e_i \prod_{j\in S^-} (e_j-1) - \prod_{i\in S^+} (e_i-1) \prod_{j\in S^-} e_j  \\
    \\
      &\text{if } \displaystyle \bigcap_{i\in S^+} H_i ^+ \cap \bigcap_{j\in S^-} H_j ^- = \varnothing \text{ and } \bigcap _{i\in (S^+ \cup S^-)} H_i \neq \varnothing .\\
  \end{array}\]
\end{prop}

\begin{rem}
  Note that only families $(1)$ and $(2)$ are necessary to generate $\mathcal{I}_1$. However, in the case that $\mathcal{A}$ is central, only families $(1)$ and $(3)$ are necessary to generate $\mathcal{I}_1$ \cite[Thm 4.5]{VG}. We include family $(3)$ as this is the presentation we will get when we set the equivariant parameter $u=1$ in our presentation of the equivariant cohomology ring.
\end{rem}

Let $P^k\subseteq H^0(M_1(\mathcal{A}))$ be the space of functions representable by polynomials in $\set{x_i}_{i=1}^n$ of degree less than or equal to $k$. Varchenko and Gelfand show\footnote{They show this only for central arrangements, but the proof extends easily to the affine case.} that the associated graded algebra of $H^0(M_1(\mathcal{A}))$ is isomorphic as a graded vector space to $H^*(M_2(\mathcal{A}))$ \cite[Cor 2.2]{VG}.

However, as $H^*(M_2(\mathcal{A}))$ is non-commutative, it cannot be isomorphic as a ring. Proudfoot shows that one obtains an isomorphism of rings if one works with coefficients in $\mathbb{F}_2$ rather than $\Q$, and he gives an equivariant cohomology interpretation of this fact in the discussion after Theorem 3.1 \cite{Pr}. This is not so satisfying, however, because the signs in the presentation of $H^*(M_2(\mathcal{A}))$ are subtle, and all of this structure is lost by working over $\mathbb{F}_2$. In this paper, we'll show that the right space to look at is $M_3(\mathcal{A})$ rather than $M_2(\mathcal{A})$. More precisely, we prove the following theorem.

\begin{thm}\label{mainresult} Let $\mathcal{A}$ be a real hyperplane arrangement.

  \begin{itemize}
    \item[(a)] The associated graded algebra of $H^0(M_1(\mathcal{A}))$ with respect to the VG filtration is isomorphic as a graded ring to $H^*(M_3(\mathcal{A}))$, with degrees halved. (That is, $P^k/P^{k-1} \cong H^{2k} (M_3(\mathcal{A}))$.)
    \item[(b)] If $W$ is a finite group acting on $\mathcal{A}$, then $H^0(M_1(\mathcal{A}))$ and $H^*(M_3(\mathcal{A}))$ are isomorphic as $W$-representations.
  \end{itemize}
\end{thm}

\begin{rem}
  A special case of Theorem \ref{mainresult}(b) is the action of $S_n$ on the braid arrangement $\mathcal{B}_n$. In this case we recover the fact that the cohomology of the configuration space of n ordered points in $\R^3$ is isomorphic to the regular representation of $S_n$ \cite{At}.
\end{rem}

\begin{rem}
One might wonder why we are only looking at $k=1,2,$ or $3$. In \cite[5.6]{dLS} it was shown (at least for central arrangements) that $H^*(M_k (\mathcal{A})) \cong H^*(M_{k-2}(\mathcal{A}))$ for all $k>3$.
\end{rem}

Last, we extend Theorem \ref{mainresult}(a) to the setting of oriented matroids. Here, $H^*(M_3(\mathcal{A}))$ is replaced by an algebra defined by Cordovil in \cite{Co} and $H^0(M_1(\mathcal{A}))$ is replaced by an analogous algebra for oriented matroids as defined in \cite{GR}.

\section{Equivariant Cohomology}

Let $G$ be a connected Lie group acting on a space $X$. Choose a contractible space $EG$ on which $G$ acts freely. Define 
\[ X_G := EG \times _G X \]
to be the quotient of $EG \times X$ by the relation 
\[ (g\cdot e, x) \sim (e, g \cdot x). \]

When $X$ is a point, the resulting space $X_G = EG/G$ is called the classifying space, and is usually denoted $BG$. For arbitrary $X$, $X_G$ is a fiber bundle over $BG$ with fiber $X$. This fiber bundle is trivial if and only if the action of $G$ on $X$ is trivial.

\begin{defn}
  The $G$-equivariant cohomology of $X$ is defined to be the ordinary cohomology of $X_G$, i.e.
  \[ H_G ^* (X) := H^*(X_G). \]
\end{defn}

\begin{rem} 
Note that in the definition of equivariant cohomology, we have made a choice of a space $EG$. However, if we choose a different contractible space with free $G$-action, we get an isomorphic cohomology ring.  
\end{rem}

Suppose that $X$ is an $m$-dimensional oriented $G$-manifold with $Y\subseteq X$ a closed, oriented, codimension $k$ $G$-submanifold. We denote by $[Y]_G \in H_G ^{k} (X)$ the class of the codimension $k$ submanifold $Y_G \subseteq X_G$. 

\begin{example}
  The group $T=S^1$ acts freely on the contractible space $ET=S^{\infty}$ with quotient $BT=\C P^{\infty}$. Hence, $H_T ^*(pt)$ is an polynomial ring generated in degree 2. 
  
  Consider the inclusion of a point into $\C$ at the origin. If we let $T$ act on $\C$ by multiplication, then the inclusion is equivariant, and the restriction map
  \[ H_T ^*(\C) \stackrel{\sim}{\longrightarrow} H_T ^*(pt) \]
is an isomorphism. The origin in $\C$ is a closed, oriented, codimension $2$ $T$-submanifold, and we denote the image of $[0]_T$ in $H_T^*(pt)$ by $u$. This class is nonzero, and thus $H_T^*(pt) \cong \Q [u]$.
\end{example}

The ordinary cohomology of a connected space $X$ has the structure of a $\Q$-algebra coming from the induced map of the map $X\to pt$. For equivariant cohomology, we get the structure of an $H_G ^*(pt)$-algebra from the map $X_G \to BG$. We can compute $H_G^*(X)$ as a module over $H_G^*(pt)$ using the Serre spectral sequence associated to this map. 

We say that $X$ is \textbf{equivariantly formal} if this spectral sequence collapses at the $E_2$ page $H^*(BG) \otimes H^*(X)$. Equivalently, $X$ is equivariantly formal if $H_G^*(X)$ is isomorphic to $H_G^*(pt) \otimes H^*(X)$ as an $H_G^*(pt)$-module and the map to $H^*(X)$ induced by the inclusion $X\hookrightarrow X_G$ is obtained by setting all positive degree elements of $H_G ^* (pt)$ to zero. If $X$ is equivariantly formal, then any lift of any $\Q$-basis of $H^*(X)$ to $H_G^*(X)$ is an $H_G^*(pt)$-basis.

\begin{example}\label{r3example}
  Let $T=S^1$ act on $\R^3\setminus \set{0}$ by rotation about the $x$-axis. Since $\R^3\setminus \set 0$ is homotopy equivalent to $S^2$, its cohomology ring is concentrated in even degree. Since the same is true for $BT$, the Serre spectral sequence degenerates at the $E_2$ page and $\R^3 \setminus \set 0$ is equivariantly formal. 

Denote by $Z^+$ and $Z^-$ the positive and negative $x$-axes, respectively. Choose an orientation of $\R^3\setminus \set 0$, and orient $Z^+$ and $Z^-$ outward. The classes $[Z^+]_T$ and $[Z^-]_T$ represent classes in $H_T ^2 (\R^3\setminus \set 0)$ as they are codimension $2$ $T$-submanifolds of the $3$-manifold $\R^3\setminus \set 0$.
  
   Consider the projection $\pi$ of $\R^3 \setminus \set 0$ onto the second and third coordinates. This maps is equivariant, and induces the map 
  \[ H_T^*(pt) \cong H_T^*(\R^2) \stackrel{\pi^*}{\longrightarrow} H_T^*(\R^3 \setminus \set 0) .\] 
  Note that $\pi^{-1}(0) = -Z^- \sqcup Z^+$, where $-Z^-$ denotes $Z^-$ with the opposite orientation. Thus the image of $u$ is $[Z^+]_T-[Z^-]_T$. Since these two subvarieties don't intersect, we get a map 
  \[ \Q[x,y]/\ideal{xy} \to H_T ^*(\R^3 \setminus \set{0}) \]
  sending $x$ to $[Z^-]_T$ and $y$ to $[Z^+]_T$. We claim that this is in fact an isomorphism.
  
  Note that the forgetful map $H_T ^*(\R^3 \setminus \set{0}) \to H^*(\R^3\setminus \set{0})\cong H^*(S^2)$ sends the classes of $[Z^-]_T$ and $[Z^+]_T$ to the same class, and this class generates $H^*(S^2)$.  Hence, the ring generated by $[Z^-]_T$ and $[Z^+]_T$ contains the image of $H_T ^*(pt)$ and surjects onto $H^*(\R^3 \setminus \set{0})$. Since $\R^3\setminus \set 0$ is equivariantly formal, this implies that the map $\Q[x,y]/\ideal{xy} \to H_T ^*(\R^3\setminus \set 0)$ is surjective.
  
  This map is also injective by a graded dimension count. Since $\R^3 \setminus \set 0$ is equivariantly formal, we have $H_T^*(\R^3 \setminus \set 0) \cong H^*(\R^3 \setminus \set 0) \otimes \Q [u]$ as a graded $\Q [u]$-module, and we're done.
\end{example}

\begin{prop}\label{specialize}
  Let $X$ be an equivariantly formal $T$-space, and let $F=X^T$. Then
\begin{align*}
    &H_T^*(X) /\ideal{u} \cong H^*(X)\text{ and} \\
    &H_T^*(X) /\ideal{u-1} \cong H^*(F).
\end{align*}
\end{prop}

\begin{proof}
  By \cite[(6.2)(2)]{GKM}, the restriction map on equivariant cohomology induces an isomorphism
  \[ H_T^*(X)[u^{-1}] \to H_T^*(F)[u^{-1}] \cong H^*(F)[u,u^{-1}] \]
  and hence
  \[ H^*(F) \cong H_T^*(X)/\ideal{u-1} \]
  This first isomorphism follows from formality.
\end{proof}

\begin{cor}\label{assocgrad}
The ring $H^*(F)$ has a natural filtration, its Rees algebra is isomorphic to $H_T^*(X)$, and its associated graded algebra is isomorphic to $H^*(X)$. 
\end{cor}

\begin{example}
  In Example \ref{r3example}, we computed the equivariant cohomology of $\R^3\setminus \set 0$. The presentation that  $H^*(Z^- \sqcup Z^+)$ inherits from $H_T^*(\R^3\setminus \set 0)$ is the same presentation provided in \cite{VG}
  \[H^*(Z^- \sqcup Z^+) = \Q[y] / \ideal{y^2-y} . \]
  If we filter this ring by degree, then its corresponding associated graded algebra is 
  \[ \gr_{\mathrm{deg}} H^*(Z^- \sqcup Z^+) \cong \Q[y]/\ideal{y^2} \]
  which is isomorphic to $H^*(\R^3\setminus \set 0)$.
\end{example}

  Suppose that $W$ is a group that acts on $X$. If the action of $W$ commutes with the action of $T$, then $W$ also acts on $F:=X^T$. 
  
\begin{prop}\label{equiv}
   If the action of $W$ commutes with the action of $T$, then $W$ acts on $H_T ^*(X)$. If $X$ is equivariantly formal, then the isomorphisms of Proposition \ref{specialize} are $W$-equivariant.
\end{prop}

\begin{proof}
  If we let $W$ act on $ET$ trivially, then it is clear that $W$ acts on $ET \times X$. To show that the action respects the relation $(y,tx)\sim (ty,x)$ we use the facts that the actions commute and hence
  \[ w\cdot (y,tx) = (y,wtx) = (y,twx) \sim (ty,wx) = w\cdot (ty,x) . \] 
  
  Since $W$ acts on $X_T$ and the projection to $BT$ is $W$-invariant, we obtain an action of $W$ on $H_T ^*(X)$ that fixes $u$. If $X$ is formal, then $H_T^*(X)$ is a flat family of $W$-representations and the result follows.
\end{proof}

\begin{rem}\label{semisimple}
  If $W$ is a finite group, the category of $W$-representations is semisimple and hence $W$-representations are completely reducible. Hence, $H^*(X)$ and $H^*(F)$ are isomorphic as $W$-representations as $\gr H^*(F) \cong H^*(X)$ and complementary subrepresentations are isomorphic to quotients.
\end{rem}

\section{Equivariant Formality}

In this section we will establish that $M_3(\A)$ is equivariantly formal. To do this we will appeal to the reasoning in Example \ref{r3example} and show that $H^*(M_3(\A))$ is concentrated in even degree.

Suppose $H_1 \in \A$, let $\A ' = \A \setminus H_1$, and let 
\[\A '' = \A^{H_1} = \set{ H \cap H_1 \mid H\in \A' , H\cap H_1 \neq \varnothing }. \]
Note that $M_3(\A') = M_3(\A) \cup M_3(\A'')$. The submanifold $M_3(\A'')$ has codimension 3 in $M_3(\A')$. 

\begin{prop}\label{evendegree}
  $H^*(M_3(\A))$ is concentrated in even degrees.
\end{prop} 

The proof of this proposition is strongly influenced by the proof of \cite[5.80,5.81]{OT} and the discussion on p. 213.

\begin{proof}
 When $\A$ consists of one hyperplane, we may reduce this to the case of a point in a line. In this case, $M_3(\A)$ is homotopic to $S^2$, and thus its cohomology is concentrated in even degree.

 Suppose this result is true for $\abs{\A}\leq n$. Suppose that $\A$ has $n+1$ hyperplanes. Inside $M_3(\A')$, $M_3(\A'')$ has a tubular neighborhood $E$ which is the disk bundle of the trivial bundle over $M_3(\A'')$. Let $E_0 = E\cap M_3(\A)$ be the punctured disk bundle. The Thom isomorphism theorem gives us that 
 \[ H^{k+1}(E, E_0) \cong H^{k-2} (E) \]
 as the Thom space for the trivial bundle is the suspension of $E \sqcup \set{pt}$ iterated 3 times.
 
 By excision $H^{k+1}(M_3(\A'),M_3(\A)) \cong H^{k+1}(E,E_0) \cong H^{k-2}(E)$. Since $E$ is homotopy equivalent to $M_3(\A'')$, we get that $H^{k-2}(E) \cong H^{k-2}(M_3(\A''))$.  The long exact sequence in cohomology gives 
 \[ \cdots \rightarrow H^k(M_3(\A')) \to H^k (M_3(\A)) \to H^{k+1}(M_3(\A'),M_3(\A)) \rightarrow \cdots  .\]
 Replacing $H^{k+1}(M_3(\A'),M_3(\A))$ with $H^{k-2}(M_3(\A''))$ and using the inductive hypothesis we get that the terms to the left and right of $H^k(M_3(\A))$ in the long exact sequence are nontrivial only when $k$ is even.
\end{proof}

\begin{cor}
  The sequence 
  \[ 0 \rightarrow H^k(M_3(\A')) \to H^k (M_3(\A)) \to H^{k-2}(M_3(\A'')) \rightarrow 0 \]
  is exact. 
\end{cor}

\begin{proof}
  Starting with the long exact sequence in the previous proof, we use that the odd degree cohomology vanishes and the result follows.
\end{proof}

From this short exact sequence we can deduce the following recursion which is also exhibited by the Orlik-Solomon algebra.

\begin{cor}\label{recursionM3}
  The Poincar\'e polynomial of $H^*(M_3(\A))$ follows the recursion
  \[ \mathrm{Poin}(M_3(\A)) = \mathrm{Poin}(M_3(\A')) + t^2 \mathrm{Poin}(M_3(\A '')) . \]
\end{cor}

Note that the $t^2$ comes from the fact that the map from $H^k (M_3(\A))$ to $H^{k-2}(M_3(\A''))$ is a degree $-2$ map.

\section{Main Results}

We first define an action of $T=S^1$ on $M_3(\mathcal{A})$. In Example \ref{r3example}, we let $T$ act by rotation about the $x$-axis, or equivalently, we thought of $\R^3$ as $\R \oplus \C$ and let $T$ act by multiplication on the complex coordinate, leaving the real coordinate fixed. We may extend this action to $V^3 \cong V \otimes(\R\oplus \C)$. Note that $\omega_{i,3}$ is $T$-equivariant and $0\in \R^3$ is $T$-fixed, so $T$ acts on \[M_3(\mathcal{A}) = \bigcap_{i=1} ^n \omega_{i,3}^{-1} (\R^3 \setminus \set 0).\] The fixed point set of this action is $M_1(\mathcal{A})$.

As $M_3(\A)$ is equivariantly formal, Corollary \ref{assocgrad} therefore yields the following result.

\begin{prop}\label{grec}
  With respect to the filtration coming from the equivariant cohomology of $M_3(\mathcal{A})$ via Corollary \ref{assocgrad},
  \[ \gr H^0(M_1(\mathcal{A})) \cong H^*(M_3(\mathcal{A})). \]
\end{prop}

The remainder of this section will be devoted to computing presentations of $H ^*(M_3(\mathcal{A}))$ and $H_T ^*(M_3(\mathcal{A}))$, which we will use to show that the filtration of $H^0(M_1(\mathcal{A}))$ coming from equivariant cohomology is the same as the one coming from Proposition \ref{VGpres}. Our calculation of $H_T^*(M_3(\mathcal{A}))$ will make use of a generating set of the ordinary cohomology of $M_3(\mathcal{A})$.

%Result About Generation

\begin{lemma}\label{generators}
Consider the map $\psi:\Q[e_1, \ldots, e_n] \to H^*(M_3(\mathcal{A}))$ taking $e_i$ to $\omega_{i,3}^*([Z^+])\in H^2(M_3(\mathcal{A}))$. This map is surjective.
\end{lemma}

\begin{rem}
In the case that $\mathcal{A}$ is central, this is established in \cite{dLS} as they provide a presentation of $H^*(M_3(\A))$. 
\end{rem}

Before we prove the result in the case that $\A$ is affine, we describe a construction called coning. Given an affine arrangement $\A = \set{H_1, \ldots , H_n}$ in a vector space $V$, the cone of $\A$ denoted $c\A$ consists of hyperplanes $\set{ cH_1, \ldots, cH_n }$ corresponding to the hyperplanes of $\A$ along with a new hyperplane $cH_0$ in $V \oplus \R$ whose defining linear form is $c\omega_0(v,r) = -r$. While any scalar multiple of this linear form would define the same hyperplane, the reasoning behind the choice of $-r$ will be apparent when we start the computation. If $H_i = \omega_i^{-1}(0) = \overline{\omega_i}^{-1}(a_i)$, then $cH_i := \set{ (v,r) \mid \overline{\omega_i}(v)=ra_i}$. Let $c\omega_i : V \times \R \to \R$ be defined by 
\[ c\omega_i (v,r) = \overline{\omega_i}(v)-ra_i .\]
Thus $cH_i = c\omega_i ^{-1}(0)$. The resulting arrangement $c\A$ is central.

Note that there is a natural inclusion of $M_1(\A)$ into $M_1(c\A)$ by $v \mapsto (v,1)$. Similarly, we get an inclusion of $M_3(\A)$ into $M_3(c\A)$ given by 
\[i: (v_1, v_2, v_3) \mapsto (v_1, v_2, v_3, (1,0,0)). \] 
Note that $c\omega_{0,3} ^{-1}(Z^+)$ doesn't intersect the embedding of $M_3(\A)$ into $M_3(c\A)$ as $c\omega_{0,3} ^{-1}(Z^+)$ contains everything of the form $(v_1,v_2,v_3, -r,0,0)$ in $M_3(c\A)$ such that $r>0$. 

\begin{proof}
  To prove Lemma \ref{generators}, we will show the following. The map 
    \[ i^* : H^*(M_3(c\A)) \to H^*(M_3(\A)) \]
    induced by the inclusion $i: M_3(\A) \to M_3(c\A)$ is surjective and the class $c\omega_{0,3}^*([Z^+])$ is in the kernel.

  First, note that $c(\A')=(c\A)'$ where the right hand side is constructed by removing the coned version $cH_1$ of the distinguished hyperplane $H_1$ in $\A$. Also, if $\A'' = \A^{H_1}$, then $c(\A'') = (c\A)'':= (c\A)^{cH_1}$ (up to differences in multiplicities). Thus we get a short exact sequence 
  \[ 0 \to H^*(M_3(c(\A'))) \to H^*(M_3(c\A)) \to H^*(M_3(c(\A''))) \to 0 . \]
  
  Suppose that $\A$ consists of one hyperplane. Then $c\A$ consists of two intersecting hyperplanes. In this scenario, 
  \[ H^*(M_3(\A)) \cong \Q[\omega_{1,3} ^*([Z^+])]/\langle \omega_{1,3} ^*([Z^+]) ^2 \rangle \]
  and 
  \[ H^*(M_3(c\A)) \cong \Q[c\omega_{1,3} ^*([Z^+]),c\omega_{0,3} ^*([Z^+])] / \langle c\omega_{1,3} ^*([Z^+]) ^2 , c\omega_{0,3} ^*([Z^+]) ^2 \rangle. \] 
  Since $c\omega_{0,3} ^{-1}(Z^+)$ doesn't intersect the embedding of $M_3(\A)$ into $M_3(c\A)$ there is a surjective map $i^*$ given by
  \[ i^*(c\omega_{1,3} ^*([Z^+])) = \omega_{1,3} ^*([Z^+]), \, i^*(c\omega_{0,3} ^*([Z^+])) = 0 \]
  as
  \[i^{-1}(c\omega_{1,3} ^{-1}(Z^+)) = \omega_{1,3} ^{-1}(Z^+) . \]
  
  Suppose that this is true for arrangements with $k\leq n$ hyperplanes. There is a commutative diagram 
    \begin{diagram}
     0 & \rTo  &H^*(M_3((c\A)')) & \rTo & H^*(M_3(c\A)) &  \rTo & H^{*-2}(M_3((c\A)'')) & \rTo & 0 \\
       &          & \dTo           &          &   \dTo {i^*}                                		& 	   & \dTo \\
    0 & \rTo & H^*(M_3(\A')) & \rTo & H^*(M_3(\A))& \rTo & H^{*-2}(M_3(\A''))  & \rTo & 0 & .
  \end{diagram} 
  The left square commutes because it is induced from inclusions of spaces. The right square commutes by naturality of the Thom isomorphism. 
  
  By our inductive hypothesis, the downward maps on the left and right are surjective. By the short five lemma, the map $i^* : H^*(M_3(c\A)) \to H^*(M_3(\A))$ is also surjective. The statement about the kernel follows from the inductive hypothesis and the commutativity of the diagram.

This tells us that, not only is there a map 
$\psi:\Q[e_1, \ldots, e_n] \to H^*(M_3(\mathcal{A}))$ taking $e_i$ to $\omega_{i,3}^*([Z^+])\in H^2(M_3(\mathcal{A}))$, but also that this map is surjective. 
\end{proof}

% Showing the recursion established by the about to be cohomology ring

In the following result, we give a proof that an algebra which we will call $B(\A)$ satisfies the same recursion on Poincar\'e polynomials that $H^*(M_3(\A))$ satisfies. This recursion will be used in the proof of Theorem \ref{tpres} to establish that $B(\A) \cong H^*(M_3(\A))$.

\begin{lemma}\label{recursion}
    Let $B(\A) := \Q[e_{H_1}, \ldots, e_{H_n}] / \mathcal{I}_0$ with $\deg e_{H_i} = 2$ where $\mathcal{I}_0$ is generated by the following families of relations:
  \[\begin{array}{l l}
      1) &e_{H_i}^2 \text{ for } i\in \set{1,\ldots, n} \\
      \\
      2) & \displaystyle \prod_{i\in S} e_{H_i}  \text{ if }\bigcap_{i\in S} H_i = \varnothing \\
      \\
      3) & \displaystyle \sum_{k\in S^{-}}\left( \prod_{i\in S^+} e_{H_i} \times \prod_{k\neq j\in S^{-}} e_{H_j}  \right) - \sum_{k\in S^+}\left(  \prod_{k\neq i\in S^+} e_{H_i} \times \prod_{j\in S^{-}} e_{H_j} \right)  \\
      \\
      &\text{if } \displaystyle \bigcap_{i\in S^+} H_i ^+ \cap \bigcap_{j\in S^-} H_j ^- = \varnothing \text{ and } \bigcap _{ i\in (S^+ \cup S^-)} H_i \neq \varnothing .\\
      \\
    \end{array}\]  
    The Poincar\'e polynomial of the algebra $B(\A)$ satisfies the relation 
    \[ \mathrm{Poin}(B(\A)) = \mathrm{Poin}(B(\A')) + t^2 \mathrm{Poin}(B(\A '')) . \]
\end{lemma}

\begin{proof}
  This is shown in \cite[2.7]{Co} in the case that $\A$ is a central arrangement. Additionally, it is shown that $B(\A)$ has a ``no broken circuit" basis. 
  
  First we will construct a complex 
  \[ 0 \to B(\A)\stackrel{t}{\longrightarrow}  B(c\A) \stackrel{\varphi}{\longrightarrow} B(\A) \to 0. \]
  Eventually, we will show that in fact this is a short exact sequence. First we need to define the maps.
  Let $S$ be a collection\footnote{If $S$ is an empty collection, then let $e_S = 1$.} of hyperplanes in $\A$, let $e_{S} = \prod_{H\in \Phi} e_H$, let $cS = \set{cH}_{H\in S}$, and let 
  \[\Phi: \Q[e_{cH_0},e_{cH_1}, \ldots, e_{cH_n}] / \langle e_{cH_i}^2 \rangle \to  \Q[e_{H_1}, \ldots, e_{H_n}] / \langle e_{H_i}^2 \rangle \] 
  be given by 
  \[ \Phi(e_{cS}) = e_S,\, \Phi(e_{cH_0} e_{cS}) = 0 .\]
  To show that this descends to a map $\varphi:B(c\A) \to B(\A)$, we need to show that generators of families 2 and 3 in $\mathcal I _0 (c\A)$ map to $\mathcal I_0 (\A)$. Note that since $c\A$ is central, families 1 and 3 generate $\mathcal I _0 (c\A)$. Family 3 is addressed in \cite[3.47, 3.49]{OT}. The map $t$ is also similarly addressed by \cite[3.47, 3.50]{OT} and sends $e_S$ to $e_{cH_0} e_{cS}$.
  
  Next, we will show that as vector spaces $B(\A) \cong BC(\A)$, the broken circuit module as defined in \cite{OT}. To define $BC(\A)$, we first need to impose an ordering on $\A$. We call a collection of hyperplanes dependent if $\bigcap_{i\in S} H_i \neq \varnothing$ and if there exists a decomposition of $S$ into $S^+ \sqcup S^-$ so that 
  \[ \bigcap_{i\in S^+} H_i ^+ \cap \bigcap_{j\in S^-} H_j ^- = \varnothing. \]
  This condition is equivalent to the defining linear forms being linearly dependent. We call a collection of hyperplanes a circuit if it is minimally dependent. A collection $S$ is called a broken circuit if adding on a hyperplane $H$ greater than all the hyperplanes in $S$, with respect to the ordering imposed above, results in a circuit. The algebra $BC(\A)$ is defined to the the free $\Q$-module generated by $1$ and the set $\set{ e_S }$ where the collections $S$ intersect nontrivially and don't contain any broken circuits. This is known as a ``no broken circuit" or NBC basis. 
  
  We already know that the natural map from $BC(c\A)$ to $B(c\A)$ induces an isomorphism of graded vector spaces in the case when $\A$ is central \cite[2.8]{Co}. To prove this in the affine case, we construct a commutative diagram \cite[3.55]{OT}
  \begin{diagram}
  0 & \rTo & BC(\A) & \rTo & BC(c\A) & \rTo & BC(\A) & \rTo & 0 \\
  & & \dTo & & \dTo & & \dTo \\
  0 & \rTo & B(\A) & \rTo & B(c\A) & \rTo & B(\A) & \rTo & 0 
  \end{diagram}
  with the top row known to be exact. Following the same reasoning as \cite[3.55]{OT}, since the middle vertical map $BC(c\A) \to B(c\A)$ is known to be an isomorphism of vector spaces, then by a diagram chase, $BC(\A) \cong B(\A)$. 
  
  Finally, we conclude that since
    \[ \mathrm{Poin}(B(c\A)) = (1+t^2) \mathrm{Poin}(B(\A)) \]
  and 
    \[ \mathrm{Poin}(B(\Hh)) = \mathrm{Poin}(B(\Hh')) + t^2 \mathrm{Poin}(B(\Hh '')) \]
  where $\Hh$ is a central arrangement, then 
    \[ \mathrm{Poin}(B(\A)) = \mathrm{Poin}(B(\A')) + t^2 \mathrm{Poin}(B(\A '')) . \]
\end{proof}

In the next result we will compute a presentation of the $T$-equivariant cohomology of $M_3(\A)$ and along the way, we will prove that $B(\A) \cong H^*(M_3(\A))$. 

\begin{thm}\label{tpres}
  Consider the map $\psi:\Q[e_1, \ldots, e_n,u] \to H_T ^*(M_3(\mathcal{A}))$ taking $e_i$ to $\omega_{i,3}^*([Z^+]_T)\in H_T^2(M_3(\mathcal{A}))$, and let $\mathcal{I}$ be the kernel. This map is surjective, and $\mathcal{I}$ is generated by the following families of relations:
\[\begin{array}{l l}
      1) &e_i (e_i - u) \text{ for } i\in \set{1,\ldots, n} \\
      \\
      2) &\displaystyle \prod_{i\in S^+} e_i \times \prod_{j\in S^{-}} (e_j - u)\text{ if }\bigcap_{i\in S^+} H_i ^+ \cap \bigcap_{j\in S^-} H_j ^- = \varnothing \\
      \\
      3) &\displaystyle u^{-1} \left( \prod_{i\in S^+} e_i \times \prod_{j\in S^{-}} (e_j - u) - \prod_{i\in S^+} (e_i-u) \times \prod_{j\in S^{-}} e_j \right) \\
      \\
      &\text{if } \displaystyle \bigcap_{i\in S^+} H_i ^+ \cap \bigcap_{j\in S^-} H_j ^- = \varnothing \text{ and } \bigcap _{i\in (S^+ \cup S^-)} H_i \neq \varnothing . \\
      \\
    \end{array}\]
\end{thm}

\begin{rem}
  Note that the expression inside the parentheses in family (3) is a multiple of $u$, so the whole thing is a polynomial.
\end{rem}

\begin{proof} 
  By Proposition \ref{evendegree}, $H^*(M_3(\mathcal{A}))$ is concentrated in even degrees, so $M_3(\mathcal{A})$ is formal. By formality, any lift of any generating set for $H^*(M_3(\mathcal{A}))$ is a generating set for $H_T^*(M_3(\mathcal{A}))$ over $H_T^*(pt)$. Hence by Lemma \ref{generators}, $H_T^*(M_3(\mathcal{A}))$ is generated by $e_1, \ldots, e_n$ and $u$. 

  The class $\omega_{i,3} ^*([Z^+]_T)$ is represented by the oriented $T$-submanifold 
  \[ Y_i ^+ = \omega_{i,3} ^{-1} (Z^+). \]
  Let $u\in H_T ^2(M_3 (\mathcal{A}))$ be the image of the generator of $H_T ^2(pt)$. By functoriality, we have $u= \omega_{i,3} ^*(u)$ for all $i$. 
 
  Recall from Example \ref{r3example} that $[Z^-]_T = y-u \in H_T ^*(\R ^3 \setminus 0)$. Hence, 
  \[ e_i - u = \omega_{i,3} ^* (y-u) \in H_T ^* (M_3 (\mathcal{A})) \]
  is represented by the oriented $T$-submanifold 
  \[ Y_i ^- = \omega _{i,3} ^{-1} (Z^-). \]

  The next step is to check that the three families belong to $\mathcal{I}$. For all $i$, $\psi (e_i (e_i - u))=0$ because $\psi(e_i (e_i -u)) = \omega_{i,3} ^*(y(y-u))=\omega_{i,3} ^*(0) = 0$. 
  
  For the second family of relations, we need to show that if
  \[ \bigcap_{i\in S^+} H_i ^+ \cap \bigcap_{j\in S^-} H_j ^- = \varnothing \]
  then 
  \[ \bigcap_{i\in S^+} Y_i ^+ \cap \bigcap_{j\in S^-} Y_j ^- = \varnothing . \] 
  Let $\pi_1: V^3 \longrightarrow V $ be the projection onto the first coordinates. That is, it should restrict the 3-arrangement complement to the real arrangement complement.
  If $p \in \bigcap_{i\in S^+} Y_i ^+ \cap \bigcap_{j\in S^-} Y_j ^-$, then $\omega_{i,3} (p) \in Z^+$ for all $i\in S^+$ and $\omega_{j,3} (p) \in Z^-$ for all $j\in S^-$. Then 
  \[ \pi_1(p) \in \bigcap_{i\in S^+} H_i ^+ \cap \bigcap_{j\in S^-} H_j ^- \]
  and hence the intersection is not empty.
  
  For the third family, since $H_T ^*(M_3(\mathcal{A}))$ is free over $H_T^*(pt)\cong \Q[u]$, we only need to show that 
  \begin{equation*}\tag{\textasteriskcentered}  \left( \prod_{i\in S^+} e_i \times \prod_{j\in S^{-}} (e_j - u) - \prod_{i\in S^+} (e_i-u) \times \prod_{j\in S^{-}} e_j \right) =0 \end{equation*}
  if $\bigcap_{i\in S^+} H_i ^+ \cap \bigcap_{j\in S^-} H_j ^- = \varnothing$ and $\bigcap _{i\in S} H_i \neq \varnothing$. Choose a point $p\in \bigcap _{i\in S} H_i$. The involution of $V$ given by reflection through $p$ takes $H_i ^+$ to $H_i ^-$ and vice versa for all $i\in S$, hence we also have 
  \[ \bigcap_{i\in S^+} H_i ^- \cap \bigcap_{j\in S^-} H_j ^+ = \varnothing . \]
  Taking the difference of the corresponding relations from family 2) yields the relation (\textasteriskcentered). We have now shown that there is a surjective map \[ \C[e_1, \ldots, e_n, u]/\mathcal{I} \to H_T^*(M_3(\mathcal{A})).\] 
  
  Setting $u$ to $0$ in $\mathcal I$ gives us the ideal $\mathcal{I}_0$ from Lemma \ref{recursion}. Thus $\mathcal{I}_0$ is contained in the kernel of the map $\Q[e_1, \ldots, e_n] \to H^*(M_3(\mathcal{A}))$. Therefore, there is a map, which is surjective by Lemma \ref{generators}, from the algebra $B(\A)$ to $H^*(M_3(\A))$. By Corollary \ref{recursionM3} and Lemma \ref{recursion}, these algebras have the same Poincar\'e polynomials and must be isomorphic.  Hence we have the following commutative diagram with exact rows and surjective columns.
  
  \begin{diagram}
     &  & \mathcal{I} & \rTo & \Q[e_1,\ldots, e_n, u] & \rTo{\psi} & H_T ^*(M_3(\mathcal{A})) & \rTo & 0 \\
       &          & \dTo{\phi}{u=0}            &          &   \dTo{\phi}{u=0}                                  		& 	   & \dTo{u=0} \\
    0 & \rTo & \mathcal{I}_0 & \rTo & \Q[e_1, \ldots, e_n] & \rTo & H^*(M_3(\mathcal{A}))  & \rTo & 0 
  \end{diagram}
  
  We would like to prove that $\mathcal{I} = \ker (\psi)$. Assume not, and let $a\in \ker{\psi} \setminus \mathcal{I}$ be a homogeneous class of minimal degree. Since $\psi(a)=0$, $\phi(a) \in \mathcal{I}_0$. Hence, there is a $b\in \mathcal{I}$ so that $\phi(a-b)=0$. Hence, $a-b=cu$ for some $c\in  \Q[e_1,\ldots, e_n, u]$. By formality, since $cu\in \ker(\psi)$, then $c\in \ker(\psi)$. Since $a-b=cu$, then $cu\notin \mathcal{I}$, and hence $c\notin \mathcal{I}$. This contradicts the assumption that $a$ is of minimal degree in the set $\ker(\psi) \setminus \mathcal{I}$. 
\end{proof}

As a byproduct of this proof, we obtain a presentation for the ordinary cohomology of $M_3(\A)$.

\begin{cor}\label{ordpres}
  Let $B(\A)$ be the algebra defined in Lemma \ref{recursion}. Then
  \[ B(\A) \cong H^*(M_3(\A)) . \]
\end{cor}

\begin{rem}\label{ordpresrem}
If $\mathcal{A}$ is central, then families (1) and (3) in Lemma \ref{recursion} are sufficient to generate $\mathcal{I}_0$. This was proven for central arrangements in \cite[5.5]{dLS}\footnote{Family 1) was omitted in error in \cite[5.5]{dLS}}. Note the similarity to the presentation of $H^*(M_2(\mathcal{A}))$, which appears in \cite[2 \& 5.2]{OS} for central arrangements and in \cite[3.45 \& 5.90]{OT} for affine arrangements.
\end{rem}

\begin{rem}\label{tpresrem}
By Proposition \ref{specialize}, if we set $u=1$ we get a presentation for $H^0(M_1(\mathcal{A}))$. This is precisely the presentation given in Proposition \ref{VGpres}, and the degree filtration is precisely the filtration defined using Heaviside functions. In particular, this demonstrates that the filtration introduced in \cite{VG} is very natural from the point of view of equivariant cohomology. Combining this observation with Proposition \ref{grec}, we obtain our main result.
\end{rem}

\begin{cor}\label{assocgr}
  The associated graded algebra of $H^*(M_1(\mathcal{A}))$ with respect to the VG filtration is isomorphic to $H^*(M_3(\mathcal{A}))$.
\end{cor}

\section{Oriented Matroids}

As a central hyperplane arrangement gives rise to a matroid, one can ask if these results generalize to matroids. More precisely, a real central hyperplane arrangement with choices of linear forms determines an oriented matroid.

\begin{defn}
  Let $E$ be a set. A signed subset $X$ of $E$ is a set $\underline{X}\subset E$ together with a partition $(X^+,X^-)$ of $\underline{X}$ into two distinguished subsets. The set $\underline{X} = X^+ \cup X^-$ is the support of $X$. The opposite of a signed set $X$, denoted by $-X$, is the signed set with $(-X)^+ = X^-$ and $(-X)^-=X^+$.
\end{defn}

Alternatively, one could consider a signed subset of $E$ to be a map of sets $\Phi_X: E \to \set{-1,0,1}$ with $X^+=\Phi_X^{-1}(1)$, and $X^-=\Phi_X^{-1}(-1)$. 

\begin{defn}
  A collection $\mathcal{C}$ of signed subsets of a set $E$ is the set of signed circuits of a loop-free oriented matroid on $E$ if and only if it satisfies the following axioms:
  \begin{enumerate}
    \item for all $X \in \mathcal{C}$, $\abs{\underline{X}}>1$ 
    \item $\mathcal{C} = -\mathcal{C}$
    \item for all $X,Y\in \mathcal{C}$, if $\underline{X} \subseteq \underline{Y}$, then $X=Y$ or $X=-Y$
    \item for all $X,Y\in \mathcal{C}$, $X\neq -Y$, and $e\in X^+ \cap Y^-$ there is a $Z\in \mathcal{C}$ such that 
    \[ Z^+ \subseteq (X^+ \cup Y^+)\setminus \set{e} \text{ and } Z^- \subseteq (X^- \cup Y^-)\setminus \set{e}. \]
  \end{enumerate}
\end{defn}

\begin{example}
  Let $\mathcal{A}= \set{H_1, \ldots, H_n}$ be a central hyperplane arrangement determined by linear forms $\set{\omega_1, \ldots, \omega_n}$, and let $E=\set{1,\ldots, n}$. Let 
  \[ H_i ^+ = \set{ v \mid \omega_i (v) > 0} \]
  and 
  \[ H_i ^- = \set{ v \mid \omega_i (v) < 0}. \]
  A signed circuit is a minimal signed subset $X$ such that
  \[ \left( \bigcap_{i \in X^+} H_i ^+ \right) \bigcap \left( \bigcap_{j\in X^-} H_j ^- \right) = \varnothing . \]
  We denote by $\mathcal{C}(\mathcal{A})$ the collection of signed circuits of the oriented matroid determined by a hyperplane arrangement $\mathcal{A}$.
\end{example}

\begin{defn} \cite{GR}
Let $(E,\mathcal{C})$ be an oriented matroid. The algebra $P(\mathcal{C})$ is defined to be the quotient of the polynomial ring $\Q[x_a]_{a\in E}$ by the ideal generated by the following elements:
\begin{enumerate}
  \item $x_a^2 - x_a$ for all $a\in E$.
  \item $\left(\prod _{a\in X^+} x_a \right) \left(\prod_{a\in X^-} (1-x_a) \right)$ for all $X\in \mathcal{C}$.
  \item $\left(\prod _{a\in X^+} x_a \right) \left(\prod_{a\in X^-} (1-x_a) \right) - \left(\prod _{a\in X^-} x_a \right) \left(\prod_{a\in X^+} (1-x_a) \right)$\\ for all $X\in \mathcal{C}$.
\end{enumerate} 
\end{defn}

\begin{rem}
  We see that in fact only families (1) and (2) are required to generate $P(\mathcal{C})$ from axiom (2) of the definition of oriented matroid. However, as in the case of Proposition \ref{VGpres}, we may replace family (2) with family (3).
\end{rem}

\begin{example}
By Proposition \ref{VGpres},
\[ P(\mathcal{C}(\mathcal{A})) \cong H^0(M_1(\mathcal{A})). \]
Filtering $P(\mathcal{C})$ by degree generalizes the VG filtration.
\end{example}

Analogous to the above construction, Cordovil defines a commutative algebra associated to an oriented matroid that is isomorphic to $H^*(M_3(\mathcal{A}))$ in the case that the oriented matroid is represented by a hyperplane arrangement.

\begin{defn}
  \cite{Co} Consider the map
  \[ \widetilde{\partial} : \mathcal{C} \to \Z [E]  \hspace{.3in} X \mapsto \sum _{a\in \underline{X}} \mathbf{\Phi_X} (a) \prod_{\stackrel{b\in \underline{X}}{b\neq a}} x_b \]
  where $\underline{X} = \set {i_1, \ldots, i_m}, i_1 < \cdots < i_m$, and $X\in \mathcal{C}$ is the signed circuit supported by $\underline{X}$ so that $\Phi_X (i_1) = 1$. The algebra $\mathbb{A}(\mathcal{C})$ is defined to be the quotient of the polynomial ring $\Q[x_a]_{a\in E}$ by the ideal generated by the following elements:
  \begin{enumerate}
    \item $x_a ^2 = 0$
    %%\item $e_i = 0$ if $i$ is a loop of $M$ ???
    \item $\widetilde{\partial} (X) = 0 $ for all $X\in \mathcal{C}$.
  \end{enumerate}
\end{defn}

\begin{example}
  By Corollary \ref{ordpres} and Remark \ref{ordpresrem},
  \[ \mathbb{A}(\mathcal{C}(\mathcal{A})) \cong H^*(M_3(\mathcal{A})). \]
\end{example}

\begin{thm}\label{matassocgr}
  With respect to the filtration of $P(\mathcal{C})$ by polynomial degree in the generators, the associated graded algebra of $P(\mathcal{C})$ is isomorphic to $\mathbb{A}(\mathcal{C})$.
\end{thm}

\begin{proof}
Given $X\in \mathcal{C}$, family (3) of $P(\mathcal{C})$ becomes
  \[ \left(\prod _{a\in X^+} x_a \right) \left(\prod_{a\in X^-} (1-x_a) \right)- \left(\prod _{a\in X^-} x_a \right) \left(\prod_{a\in X^+} (1-x_a) \right)  \]
  \[ = \widetilde{\partial}(X) + \text{ lower degree terms.} \]
This tells us that there is a natural surjection of rings $\mathbb{A}(\mathcal{C}) \to \gr P(\mathcal{C})$. By \cite[2.8]{Co}, the dimensions of the two algebras are equal. Hence this map is an isomorphism.
\end{proof}

  Note that if $\mathcal{A}$ of a central hyperplane arrangement and $\mathcal{C} = \mathcal{C}(\mathcal{A})$, this agrees with our earlier result, Corollary \ref{assocgr}. There are several advantages of the approach taken in Section 3. Proposition \ref{grec} shows \emph{a priori} that the cohomology ring of $M_1(\mathcal{A}) = M_3 (\mathcal{A}) ^T$ admits a filtration whose associated graded algebra is isomorphic to the cohomology ring of $M_3(\mathcal{A})$ without having to work directly with any presentations. Furthermore, Remark \ref{tpresrem} motivates the appearance of Heaviside functions in the work of Varchenko and Gelfand. 
  
  \begin{rem}
  In a future paper we will generalize Theorem \ref{matassocgr} to pointed oriented matroids, which are the combinatorial analogues of non-central arrangements.
  \end{rem}
  
\begin{question}
In \cite{GR} and \cite{Pr}, when extending the Orlik-Solomon algebra to the setting of oriented matroids, the Salvetti complex is used to give this a geometric realization. In particular, when the oriented matroid is representable by a hyperplane arrangement, the Salvetti complex is a deformation retract of the complexified hyperplane complement. Is there such a complex that one can construct for the corresponding $3$-arrangement?
\end{question}

\section{Finite Group Actions}

Proposition \ref{grec} has applications for the representation theory of finite groups. Let $\mathcal{A}$ be a real hyperplane arrangement.

\begin{prop}\label{regrep}
  Suppose that a finite group $W$ acts on $\mathcal{A}$. Then $W$ acts on both $M_1(\mathcal{A})$ and $M_3(\mathcal{A})$, and $H^*(M_3(\mathcal{A})) \cong H^*(M_1(\mathcal{A}))$ as $W$-representations.
\end{prop}

\begin{proof}
  We will first show that the actions of $W$ and $T$ on $M_3(\mathcal{A})$ commute. The group $T$ acts trivially on the first coordinate of $M_3(\mathcal{A})\subset V^3$, so we only need to show that the actions commute on the last two coordinates. If we treat the last two coordinates as one complex coordinate, we see that $T$ acts by scalar multiplication. Hence, these two actions commute. 
  
  By Proposition \ref{equiv}, $H^*(M_3(\mathcal{A})) \cong \gr H^*(M_1(\mathcal{A}))$ as $W$-representations. The result now follows from Remark \ref{semisimple}. 
\end{proof}

The most interesting class of examples we have are Weyl groups acting on Coxeter arrangements.

\begin{example}
  Let $\mathcal{A}$ be a Coxeter arrangement with Weyl group $W$. In this case $W$ acts simply transitively on the chambers of $\mathcal{A}$, so $H^*(M_3(\mathcal{A})) \cong_W H^0(M_1(\mathcal{A}))$ is the regular representation. 
\end{example}

\begin{rem}
  When $\mathcal{A} = \mathcal{B}_n$ is the braid arrangement and $M_3(\mathcal{A})$ is the configuration space of $n$ points in $\R^3$, this is already known \cite{At}.
\end{rem}

\begin{example}
  Consider the arrangement $\mathcal{B}_4$ whose Weyl group is the symmetric group $S_4$. Computing each graded component as an $S_4$-representation we get 
  \begin{align*}
  & H^0(M_3(\B_4)) \cong \tau \\
  & H^2(M_3(\B_4)) \cong \rho + \Lambda^2(\rho) \\
  & H^4(M_3(\B_4)) \cong \rho + \Lambda^2(\rho) + 2\omega + \sigma \\
  & H^6(M_3(\B_4)) \cong \rho + \Lambda^2(\rho)
  \end{align*}
  where $\tau$ corresponds to the partition $(4)$, $\rho$ corresponds to the partition $(3,1)$, $\Lambda^2(\rho)$ corresponds to the partition $(2,1,1)$, $\omega$ corresponds to the partition $(2,2)$, and $\sigma$ corresponds to the partition $(1,1,1,1)$. Adding up the representations of all of the graded components yields the regular representation of $S_4$.
\end{example}

\begin{example}
  The semi-order arrangement $\mathcal{S}_3$ is defined by the affine linear forms 
  \[ \omega_{ij} = x_i - x_j-1, \hspace{.1in} 1\leq i,j \leq 3 , \hspace{.1in} H_{ij} = \omega_{ij}^{-1} (0) .  \]
  
  \begin{figure}[htp]
  \begin{center}
    \includegraphics[width=2in]{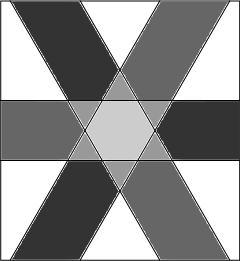}
    \caption{The semi-order arrangement $\mathcal{S}_3$ with shaded orbits.}
  \end{center}
  \end{figure}
  
  Looking at the action of the symmetric group $S_3$ on $H^0(M_1(\mathcal{S}_3))$ we get 
  \[H^0(M_1(\mathcal{S}_3)) \cong _{S_3} 5\tau + 2\sigma + 6\rho . \]
  where $\tau$ corresponds to the partition $(3)$, $\rho$ corresponds to the partition $(2,1)$, and $\sigma$ corresponds to the partition $(1,1,1)$. 
  
  The bounded region containing the origin corresponds to a $\tau$, the remaining bounded regions correspond to $\tau+\sigma +2\rho$, and the unbounded regions correspond to $3\tau + \sigma + 3\rho$.
  
  The algebra $H^*(M_3(\mathcal{S}_3))$ decomposes as 
  \begin{align*}
    &H^0(M_3(\mathcal{S}_3)) \cong _{S_3} \tau \\
    &H^2(M_3(\mathcal{S}_3)) \cong _{S_3} \tau + \sigma + 2\rho \\
    &H^4(M_3(\mathcal{S}_3)) \cong _{S_3} 3\tau + \sigma + 4\rho 
  \end{align*}
  as $S_3$-representations.
\end{example}

\end{document}